\documentclass{article}

\usepackage{mystyle}
\usepackage{wasysym}
\begin{document}
\title{Structure theorems above a strongly compact cardinal}
\author{Gabriel Goldberg\\ Evans Hall\\ University Drive \\ Berkeley, CA 94720}
\maketitle
\begin{abstract}
While many inner model theoretic combinatorial principles are incompatible with large cardinal axioms, on some rare occasions, large cardinals actually imply that the structure of the universe of sets is analogous to the canonical inner models. This note provides two new instances of this phenomenon.
\end{abstract}
\section{Introduction}
There is a deep tension between large cardinal axioms and the strong combinatorial properties typical of canonical inner models. Many such combinatorial principles must fail at certain large cardinals; for example, Jensen 
\cite{JensenKunen} showed that \(\diamondsuit^+(\kappa)\) fails at every ineffable cardinal \(\kappa\), while \(\square(\kappa)\) fails at every weakly compact cardinal \(\kappa\) by \cite{Jensen}. Despite this inherent tension, large cardinal axioms are occasionally found to have consequences that are highly analogous to the properties of the known canonical inner models. A typical example is a theorem of Kunen \cite{JensenKunen} that states that \(\diamondsuit(\kappa)\) holds at every subtle cardinal \(\kappa\). Perhaps the most famous example is due to Solovay \cite{Solovay}: the Singular Cardinals Hypothesis holds above the first strongly compact cardinal. In this short paper, we will be concerned with two more instances of inner model-like behavior above the first strongly compact cardinal.

Given Kunen's theorem that \(\diamondsuit(\kappa)\) holds at every subtle cardinal \(\kappa\), it is natural to ask whether a weaker large cardinal notion suffices. The most basic question in this vein is whether \(\diamondsuit(\kappa)\) necessarily holds at every inaccessible cardinal \(\kappa\). In early unpublished work, Woodin used Radin forcing to answer this question negatively:
\begin{thm}[Woodin]
It is consistent that there is a strongly inaccessible cardinal \(\kappa\) such that \(\diamondsuit(\kappa)\) fails.\qed
\end{thm}
Whether \(\diamondsuit(\kappa)\) must hold at every weakly compact cardinal \(\kappa\) remains open to this day. More generally, the relationship between compactness principles and guessing principles like \(\diamondsuit\) is still an active area of study (see for example \cite{Omer}). The first theorem of this paper shows that a particularly nice compactness principle that answers these questions for all sufficiently large cardinals:
\begin{repthm}{StronglyCompactDiamond}
Suppose \(\lambda\) is a regular cardinal that lies above the least strongly compact cardinal and satisfies \(2^{<\lambda} = \lambda\). Then \(\diamondsuit(\lambda)\) holds.
\end{repthm}

Our second theorem concerns the structure of indecomposable ultrafilters. The \(\kappa\)-complete ultrafilters on \(\kappa\) are of fundamental interest in set theory, so it is natural to ask about those ultrafilters that fall just short of \(\kappa\)-complete:
\begin{defn}
A uniform ultrafilter \(U\) on a cardinal \(\kappa\) is {\it indecomposable} if whenever \(\bigcup_{\alpha < \nu} S_\alpha\in U\) for some \(\nu < \kappa\), there is a countable set \(\sigma\subseteq\nu\) with \(\bigcup_{\alpha\in \sigma}S_\alpha\in U\).
\end{defn}
Is the concept of an indecomposable ultrafilter really so different from that of a \(\kappa\)-complete ultrafilter on \(\kappa\)? The concepts are not exactly the same, since given a measurable cardinal \(\kappa\), it is easy to build indecomposable ultrafilters on \(\kappa\) that are not \(\kappa\)-complete. Moreover, there exist nonmeasurable cardinals carrying indecomposable ultrafilters: any limit of measurable cardinals of countable cofinality does. In the canonical inner models, these are the only kinds of cardinals that carry indecomposable ultrafilters. 

On the other hand, in one of the first applications of his now-famous notion of forcing, Prikry showed that it is consistent that there is an indecomposable ultrafilter on a cardinal of cofinality \(\omega\) even when there are no measurable cardinals at all. Silver \cite{Silver} then asked whether an inaccessible cardinal that carries an indecomposable ultrafilter is necessarily measurable. It was widely expected that the answer to Silver's question would be negative, and this was confirmed by work of Sheard \cite{Sheard}:
\begin{thm}[Sheard]
It is consistent that there is an inaccessible non-weakly compact cardinal \(\kappa\) that carries an indecomposable ultrafilter.\qed
\end{thm}
The main result of this paper, however, states that above the first strongly compact cardinal, Silver's question actually has a positive answer:
\begin{repthm}{MainThm}
Suppose \(\kappa\) lies above the least strongly compact cardinal and carries an indecomposable ultrafilter. Then \(\kappa\) is either measurable or a limit of measurable cardinals of countable cofinality. 
\end{repthm}
Thus once again, the structure of \(V\) above the first strongly compact cardinal resembles the structure of the canonical inner models.
\section{\(\diamondsuit\) at inaccessible cardinals}
We will prove \cref{StronglyCompactDiamond} using some observations about diamond principles on ultrafilters.
\begin{defn}
Suppose \(F\) is a filter on a cardinal \(\lambda\). 
\begin{description}
\item[\(\diamondsuit(F)\):] there is a sequence of sets \(\langle S_\alpha : \alpha < \lambda\rangle\) such that for every set \(A\subseteq \lambda\), \(\{\alpha < \lambda: A\cap \alpha =  S_\alpha\}\in F^+\).
\item[\(\diamondsuit^-(F)\):] there is a sequence of sets \(\langle \mathcal S_\alpha : \alpha < \lambda\rangle\) with \(|\mathcal S_\alpha| \leq \alpha\) for all \(\alpha < \lambda\) such that for every set \(A\subseteq \lambda\), \(\{\alpha < \lambda: A\cap \alpha\in \mathcal S_\alpha\}\in F^+\).
\end{description}
Suppose \(\lambda\) is a regular uncountable cardinal and \(S\subseteq \lambda\) is a stationary set. Let \(\mathcal C\) denote the closed unbounded filter on \(\lambda\). Then \(\diamondsuit(S)\) abbreviates \(\diamondsuit(\mathcal C\restriction S)\) and \(\diamondsuit^-(S)\) abbreviates \(\diamondsuit^-(\mathcal C\restriction S)\). 
\end{defn}
Obviously \(\diamondsuit(F)\) implies \(\diamondsuit^-(F)\). On the other hand, \(\diamondsuit(U)\) cannot hold when \(U\) is an ultrafilter, while we will see that \(\diamondsuit^-(U)\) can, so \(\diamondsuit^-(F)\) does not imply \(\diamondsuit(F)\). On the other other hand, \(\diamondsuit^-(F)\) implies \(\diamondsuit^-(G)\) for any filter \(G\subseteq F\), so if \(\lambda\) is regular and \(F\) extends the closed unbounded filter on \(\lambda\), then \(\diamondsuit^-(F)\) implies \(\diamondsuit^-(S)\), and hence \(\diamondsuit(S)\), for every \(S\in F\). Here we use the following theorem of Kunen:
\begin{thm}[Kunen]\label{KunenDiamond}
Suppose \(\lambda\) is a regular uncountable cardinal and \(F\) is a normal filter on \(\lambda\). Then \(\diamondsuit^-(F)\) and \(\diamondsuit(F)\) are equivalent.\qed
\end{thm}
We will be interested in the principle \(\diamondsuit^-(U)\) for ultrafilters \(U\) that are {\it not} normal. 
\begin{thm}\label{WeaklyNormalDiamond}
Suppose \(\lambda\) is a weakly inaccessible cardinal such that \(2^{<\lambda} = \lambda\). Suppose \(U\) is a weakly normal ultrafilter on \(\lambda\) that concentrates on ordinals \(\gamma\) such that \(2^{\textnormal{cf}(\gamma)}\leq \gamma\). Then \(\diamondsuit^-(U)\) holds.
\begin{proof}
Let \(\iota = \text{cf}^{M_U}(\sup j_U[\lambda])\). A theorem of Ketonen implies that every subset of \(M_U\) of cardinality \(\lambda\) is contained in an element of \(M_U\) of \(M_U\)-cardinality \(\iota\). Since \(|P_{\text{bd}}(\lambda)| = 2^{<\lambda} = \lambda\), there is a set \(C\in M_U\) such that \(|C|^{M_U} =\iota\) and \(j_U[P_\text{bd}(\lambda)]\subseteq C\). Let \(C^* = \{\bigcup A : A\in P^{M_U}(C)\}\). Then \(C^*\in M_U\) and \(|C^*|^{M_U} \leq (2^{\iota})^{M_U}\leq \sup j_U[\lambda]\) using Los's Theorem and the fact that \(U\) concentrates on ordinals \(\gamma\) such that \(2^{\textnormal{cf}(\gamma)}\leq \gamma\).

By construction, \(\{j_U(A)\cap \sup j_U[\lambda] : A\subseteq \lambda\}\subseteq C^*\). Fix \(\langle \mathcal A_\alpha : \alpha < \lambda\rangle\) such that \(C^* = [\langle \mathcal A_\alpha : \alpha < \lambda\rangle]_U\). Since \(|C^*| \leq \sup j_U[\lambda] = [\text{id}]_U\), by Los's Theorem, we may assume without loss of generality that \(|\mathcal A_\alpha|\leq \alpha\) for all \(\alpha < \lambda\). Moreover for any \(A\subseteq \lambda\), since \(j_U(A)\cap \sup j_U[\lambda]\in C^*\), we have that \(\{\alpha < \lambda : A\cap \alpha\in \mathcal A_\alpha\}\in U\). Thus \(\langle \mathcal A_\alpha : \alpha < \lambda\rangle\) witnesses \(\diamondsuit^-(U)\).
\end{proof}
\end{thm}

\begin{thm}\label{StronglyCompactDiamond}
Suppose \(\lambda\) is a regular cardinal with \(2^{<\lambda} = \lambda\). If some \(\delta\leq \lambda\) is \(\lambda\)-strongly compact, then \(\diamondsuit(\lambda)\) holds.
\end{thm}
\begin{proof}[Proof of \cref{StronglyCompactDiamond}]
Since \(\delta\) is \(\lambda\)-strongly compact, there is a weakly normal ultrafilter \(U\) on \(\lambda\) concentrating on ordinals of cofinality less than \(\delta\). Since \(U\) satisfies the hypotheses of \cref{WeaklyNormalDiamond}, \(\diamondsuit^-(U)\) holds. By the remarks preceding \cref{KunenDiamond}, it follows that \(\diamondsuit(A)\) holds for every \(A\in U\), and in particular, \(\diamondsuit(\lambda)\) holds.
\end{proof}

\begin{qst}
Suppose \(\kappa\) is a Mahlo cardinal above the least strongly compact cardinal. Must \(\diamondsuit(\text{Reg}\cap \kappa)\) hold?
\end{qst}
\section{Indecomposable ultrafilters and measurability}
Before turning to the proof of \cref{MainThm}, let us provide some basic definitions.
\begin{defn}
Suppose \(\lambda\) is a cardinal. An ultrafilter \(U\) is {\it \(\lambda\)-indecomposable} if any subset \(F\) of \(U\) with \(|F| \leq \lambda\) and \(\bigcap F = \emptyset\) contains a set \(G\) with \(|G| < \lambda\) and \(\bigcap G = \emptyset\).
\end{defn} 
Indecomposability can be seen as a local form of completeness: an ultrafilter \(U\) is \(\kappa\)-complete if and only if it is \(\lambda\)-decomposable for all cardinals \(\lambda < \kappa\). If \(\lambda\) is regular, then \(U\) is \(\lambda\)-indecomposable if and only if \(U\) is {\it \(\lambda\)-descendingly complete}, i.e., closed under decreasing intersections of length \(\lambda\). Perhaps the best way to think about indecomposability is in terms of the ultrapower construction:
\begin{lma}\label{DecomposableUltrapower}
Suppose \(\kappa\) is a regular cardinal. An ultrafilter \(U\) is \(\kappa\)-indecomposable if and only if \(j_U(\kappa)\) is the least upper bound of \(j_U[\kappa]\).\qed
\end{lma}

We now turn to some deeper facts about indecomposability. A key ingredient in the proof is the use of external ultrapowers of models of set theory. We begin by disambiguating some terms.
\begin{defn}
Suppose \(M\) is a model of set theory and \(X\in M\). 
\begin{itemize} 
	\item An {\it \(M\)-ultrafilter on \(X\)} is an ultrafilter over the Boolean algebra \(P^M(X)\).
	\item If \(U\) is an \(M\)-ultrafilter on \(X\) and \(\gamma\in M\), then \(U\) is {\it \(\gamma\)-complete over \(M\)} if for any \(F\subseteq U\) with \(F\in M\) and \(|F|^M < \gamma\), \((\bigcap F)^M\in U\).
	\item If \(U\) is an \(M\)-ultrafilter on \(X\) and \(\gamma\in M\), then \(U\) is {\it \(\gamma\)-amenable to \(M\)} if for any \(F\subseteq P^M(X)\) with \(|F|^M \leq \gamma\), \(F\cap U\in M\).
\end{itemize}
\end{defn}

One of the tools in the proof of our theorem is Silver's factorization theorem for indecomposable ultrafilters.
\begin{thm}[Silver]\label{SilverFactor}
Any indecomposable ultrafilter on a cardinal \(\kappa > 2^{\omega_1}\) factors as an external iterated ultrapower \((D,U)\) where \(D\) is an ultrafilter on \(\omega\) and \(U\) is an \(M_D\)-ultrafilter on \(j_D(\kappa)\) that is \(j_D(\gamma)\)-complete over \(M_D\) for all \(\gamma < \kappa\).\qed
\end{thm}
To prove \cref{MainThm}, we combine Silver's factorization theorem with the following fact, which is closely related to work of Woodin on the approximation and cover properties in the context of supercompact cardinals:
\begin{thm}\label{Absorption}
Suppose \(\delta\) is strongly compact, \(D\) is an ultrafilter on a set of cardinality less than \(\delta\), and \(U\) is an \(M_D\)-ultrafilter on \(X\) that is \(j_D(\gamma)\)-complete over \(M_D\) for all \(\gamma < \delta\). Then \(U\) belongs to \(M_D\).
\end{thm}
We will use the following basic fact, strengthenings of which are fundamental lemmas in the theory of strong compactness:
\begin{lma}\label{TrivialCovering}
Suppose \(D\) is an ultrafilter on a set of cardinality \(\gamma\). Then any \(\sigma\subseteq M_D\) with \(|\sigma| \leq \gamma\) is contained in some \(\tau\in M_D\) with \(|\tau|^{M_D} \leq j_D(\gamma)\).
\begin{proof}
Let \(\eta = |\sigma|\), and fix functions \(\langle F(\alpha) : \alpha < \eta\rangle\) such that \(\sigma = \{[F(\alpha)]_D : \alpha < \eta\}\). Since \([F(\alpha)]_D = j(F(\alpha))([\text{id}]_D) = j(F)(j(\alpha))([\text{id}]_D)\), we have \(\sigma\subseteq \{j_D(F)(\beta)([\text{id}]_D) : \beta < j_D(\eta)\}\). Letting \(\tau\) be this set, observe that \(\tau\in M_D\), \(\sigma\subseteq \tau\), and \(|\tau|^{M_D} \leq j_D(\eta) < j_D(\delta)\).
\end{proof}
\end{lma}

We use a second lemma which generalizes the ``ancient Kunen argument'':
\begin{lma}[Kunen]\label{AncientKunen}
Suppose \(U\) is an \(M\)-ultrafilter on \(X\) and \(\gamma\) is an \(M\)-cardinal. Then the following are equivalent:
\begin{enumerate}[(1)]
\item \(U\) is \(\gamma\)-amenable to \(M\).
\item For any \(A\subseteq j_U(\gamma)\) with \(A\in\textnormal{Ult}(M,U)\), \(j_U^{-1}[A]\in M\).
\end{enumerate}
\begin{proof}
(1) implies (2): Fix \(A\subseteq j_U(\gamma)\) with \(A\in\textnormal{Ult}(M,U)\). There is some \(f\) such that \(A = [f]_U\). For \(\xi < \gamma\), we have \(\xi\in j_U^{-1}[A]\) if and only if the set \(A_\xi = \{x\in X : \xi\in f(x)\}\) belongs to \(U\). But \(\gamma\)-amenability yields that \(\{A_\xi : \xi < \gamma\}\cap U\in M\). It follows that \(j_U^{-1}[A] = \{\xi : A_\xi\in U\}\in M\).

(2) implies (1): Suppose \(F\subseteq P^M(X)\) and \(|F|^M \leq \gamma\). Enumerate \(F\) in \(M\) as \(\langle S_\xi : \xi < \gamma\rangle\). Let \(\langle T_\xi : \xi < j_U(\gamma)\rangle = j_U(\langle S_\xi : \xi < \gamma\rangle)\). Define \(A = \{\xi < j_U(\gamma) : [\text{id}]_U\in T_\xi\}\). Then \(S_\xi\in U\) if and only if \([\text{id}]_U\in j_U(S_\xi) = T_{j_U(\xi)}\), or equivalently if and only if \(j_U(\xi)\in A\). By (2), \(j_U^{-1}[A]\in M\). Therefore \(F\cap U = \{S_\xi : \xi\in j_U^{-1}[A]\}\in M\).
\end{proof}
\end{lma}
As a corollary, we obtain in the general case a weakening of the usual Kunen lemma from the wellfounded context:
\begin{cor}[Kunen]\label{AncientKunen2}
Suppose \(M\) is a model of set theory, \(\gamma\) is an \(M\)-cardinal, and \(U\) is an \(M\)-ultrafilter that is \(\eta\)-complete over \(M\) where \(\eta = ((2^\gamma)^+)^M\). Then \(U\) is \(\gamma\)-amenable to \(M\).
\begin{proof}
Let \(j : M \to N\) be the ultrapower of \(M\) by \(U\). Since \(U\) is \(\eta\)-complete over \(M\) where \(|P^M(\gamma)|^M < \eta\), \(j\) restricts to a surjection from \(P^M(\gamma)\) onto \(P^N(j_U(\gamma))\). For any \(A\subseteq j(\gamma)\) with \(A\in N\), there is some \(B\in M\) with \(j(B) = A\). By elementarity, we must have \(B = j^{-1}[A]\). Hence \(j^{-1}[A]\in M\). By \cref{AncientKunen}, \(U\) is \(\gamma\)-amenable to \(M\).
\end{proof}
\end{cor}

The following is a version of Kunen's Commuting Ultrapowers Lemma for countably incomplete ultrafilters:
\begin{lma}\label{KunenCommute}
Suppose \(D\) is an ultrafilter on a set of size \(\gamma\) and \(Z\) is a \(\gamma^+\)-complete ultrafilter on a set \(X\in M_D\). Then \(Z\cap P^{M_D}(X)\) belongs to \(M_D\).\qed
\end{lma}
We omit the proof, which will appear in another paper; an elegant exposition would require introducing some terminology we would like to avoid.

\begin{proof}[Proof of \cref{Absorption}]
We first show that \(U\) generates a \(\delta\)-complete filter on \(X\). Suppose \(G\subseteq U\) and \(|G| < \delta\). We must show that \(\bigcap G\neq \emptyset\). By \cref{TrivialCovering}, \(G\) is contained in some \(F\in M_D\) with \(|F|^{M_D} < j_D(\delta)\). By \cref{AncientKunen2}, \(F\cap U\in M\). Since \(U\) is \(\delta\)-complete over \(M\), it follows that \(\bigcap (F\cap U)\in U\). But \(G\subseteq F\cap U\), so \(\bigcap G\neq \emptyset\), as desired.

Since \(\delta\) is strongly compact, \(U\) can be extended to a \(\delta\)-complete ultrafilter \(Z\). But \(U = Z\cap P^{M_D}(X)\) belongs to \(M_D\) by \cref{KunenCommute}.
\end{proof}

Using \cref{Absorption}, we can finally prove our main theorem on indecomposable ultrafilters.
\begin{thm}\label{MainThm}
If \(\kappa\) lies above the first strongly compact cardinal and carries an indecomposable ultrafilter, then \(\kappa\) is either measurable or a limit of measurable cardinals of countable cofinality. 
\end{thm}
\begin{proof}[Proof of \cref{MainThm}]
By \cref{SilverFactor}, any indecomposable ultrafilter on \(\kappa\) gives rise to an external iterated ultrapower \((D,U)\) where \(D\) is an ultrafilter on \(\omega\) and \(U\) is an \(M_D\)-ultrafilter on \(j_D(\kappa)\) that is \(j_D(\gamma)\)-complete over \(M_D\) for all \(\gamma < \kappa\). By \cref{Absorption}, \(U\) belongs to \(M_D\); in other words, this external iterated ultrapower is actually internal. 

Working in \(M_D\), let \(\eta\) be the critical point of \(U\). Since \(U\) lies on \(j_D(\kappa)\), \(\eta\leq j_D(\kappa)\). Since \(U\) is \(j_D(\gamma)\)-complete for all \(\gamma < \kappa\), \(\eta > j_D(\gamma)\) for all \(\gamma < \kappa\). In particular, \(\eta\) is uncountable in \(M_D\), so \(\eta\) is a measurable cardinal of \(M_D\). 

If \(\eta = j_D(\kappa)\), then \(\kappa\) is measurable by elementarity. Therefore assume \(\eta < j_D(\kappa)\). The usual argument implies that \(\kappa\) is a limit of measurable cardinals: for any \(\gamma < \kappa\), \(M_D\) thinks there is a measurable cardinal strictly in between \(j_D(\gamma)\) and \(j_D(\kappa)\), so by elementarity, there is a measurable cardinal strictly between \(\gamma\) and \(\kappa\). The existence of \(\eta\) implies that \(j_D(\kappa) \) is not the least upper bound of \(j_D[\kappa]\), and hence \(D\) is \(\text{cf}(\kappa)\)-decomposable by \cref{DecomposableUltrapower}. Since \(D\) is an ultrafilter on \(\omega\), the only way this can happen is if \(\text{cf}(\kappa) = \omega\).
\end{proof}

\bibliography{Bibliography.bib}
\bibliographystyle{unsrt}
\end{document}